\documentclass[a4paper,12pt]{article}
\usepackage{amsthm}
\usepackage{amsfonts}
\usepackage{amsmath}
\usepackage{amssymb}
\usepackage[pdftex]{graphicx}
\usepackage{hyperref}
\usepackage{datetime}
\usepackage{diagrams}
\usepackage{commath}
\usepackage{float}
\restylefloat{table}

\newtheorem{theorem}{Theorem}[section]
\theoremstyle{definition}

\theoremstyle{remark}

\theoremstyle{plain}
\newtheorem{lemma}[theorem]{Lemma}

\newtheorem{proposition}[theorem]{Proposition}

\newcommand{\diag}{\text{diag}}
\newcommand{\bq}{/\!\!/}

\begin{document}

\title{Quasi-positive curvature on a biquotient of $Sp(3)$}

\author{Jason DeVito and Wesley Martin}




\date{}

\maketitle

\begin{abstract}
Suppose $\phi_3:Sp(1)\rightarrow Sp(2)$ denotes the unique irreducible $4$-dimensional representation of $Sp(1) = SU(2)$ and consider the two subgroups $H_1, H_2\subseteq Sp(3)$ with $H_1 = \{\diag(\phi_3(q_1), q_1): q_1 \in Sp(1)\}$ and $H_2 = \{\diag(\phi_3(q_2),1):q_2\in Sp(1)\}$.  We show that the biquotient $H_1\backslash Sp(3)/H_2$ admits a quasi-positively curved Riemannian metric.

\end{abstract}

\section{Introduction}\label{intro}

Manifolds of positive sectional curvature have been studied extensively.  Despite this, there are very few known examples of positively curved manifolds.  In fact, other than spheres and projective spaces, every known compact simply connected manifold admitting a metric of positive curvature is diffeomorphic to an Eschenburg space \cite{Es1,AW}, Eschenburg's inhomogoneous flag manifold, the projectivized tangent bundle of $\mathbb{K}P^2$ with $\mathbb{K}\in\{\mathbb{C},\mathbb{H},\mathbb{O}\}$ \cite{Wa}, a Bazaikin space \cite{Ba1}, the Berger space \cite{Ber}, or a certain cohomogeneity one manifold which is homeomorphic, but not diffeomorphic, to $T^1 S^4$ \cite{De,GVZ}.

Because of the difficulty in constructing new examples, attention has turned to the easier problem of finding examples with quasi- or almost positive curvature.  Recall that a Riemannian manifold is said to be quasi-positively curved if it admits a non-negatively curved metric with a point $p$ for which the sectional curvatures of all $2$-planes at $p$ are positive.  A Riemannian manifold is called almost positively curved if the set of points for which all $2$-planes are positively curved is dense.  Examples of manifolds falling into either of these cases are more abundant.  See \cite{DDRW,Di1,EK,GrMe1,Ke1,Ke2,KT,PW2,Ta1,W,Wi}.

In \cite{DDRW}, the first author, together with DeYeso, Ruddy, and Wesner prove that there are precisely $15$ biquotients of the form $Sp(3)\bq Sp(1)^2$ and show that $8$ of them admit quasi-positively curved metrics.  We show that their methods can be adapted to work on a $9$th example, called $N_9$ in \cite{DDRW}.  That is, we show $N_9$ admits a metric of quasi-positive curvature as well.

To describe this example, we first set up notation.  Let $\phi_3:Sp(1) = SU(2)\rightarrow Sp(2)$ denote the unique irreducible $4$-dimensional representation of $Sp(1)$.  Further, let $G = Sp(3)$, and let $H_1 =\{ \diag(\phi_3(q_1), q_1)\in G: q_1 \in Sp(1)$ and $H_2 = \{\diag(\phi_3(q_2),1)\in G): q_2\in Sp(1)\}$.  Finally, set $H = H_1\times H_2\subseteq G\times G$.

\begin{theorem}\label{main}  The biquotient $H_1\backslash G/H_2$ admits a metric of quasi-positive curvature.
\end{theorem}

In fact, we show the metric constructed on $G$ in \cite{DDRW} is $H$ invariant and the induced metric on $N_9$ is quasi-positively curved.

Finally, we point out that one of the first steps in the proof, Proposition \ref{niceform}, does not hold for any of the remaining inhomogeneous biquotients of the form $Sp(3)\bq Sp(1)^2$.  In particular, a new approach is needed to determine whether these other biquotients admit metrics of quasi-positive curvature.

The outline of this paper is as follows.  Section $2$ will cover the necessary background, leading to a system of equations paramaterized by $p\in G$, which govern the existence of a zero curvature plane at $[p^{-1}]\in G\bq H$.  In Section 3, we find a particular point $p\in G$ for which there are no non-trivial solutions to the system of equations, establishing Theorem \ref{main}.

This research was partially supported by the Bill and Roberta Blankenship Undergraduate Research Endowment; we are grateful for their support.

\section{Background}\label{background}

We will use the setup of \cite{DDRW}.  As the calculations will be done on the Lie algebra level, we now describe all the relevant Lie algebras.

We recall the Lie algebra $\mathfrak{sp}(n)$ consists of all $n\times n$ quaternionic skew-Hermitian matrices with Lie bracket given by the commutator.  That is, $\mathfrak{sp}(n) =\{A\in M_n(\mathbb{H}): A + \overline{A}^t = 0\}$, where $\mathbb{H}$ denotes the skew-field of quaternions, and the Lie bracket is given by $[A,B] = AB - BA$.  When $n = 1$, this Lie algebra is simply $\operatorname{Im}\mathbb{H}$.

Then the Lie algebra of $G = Sp(3)$, $\mathfrak{g} = \mathfrak{sp}(3)$ consists of the $3\times 3$ skew-Hermitian matrices over $\mathbb{H}$.  Further, we set $K = Sp(2)\times Sp(1)$, block diagonally embedded into $G$ via $(A,q)\mapsto \diag(A,q)\in G $.  Then one easily sees that $\mathfrak{k} = \mathfrak{sp}(2)\oplus \mathfrak{sp}(1)$ is embedded into $\mathfrak{g}$ via $(B,r)\mapsto \diag(B,r)$.

We also use the description of $\phi_3$ on the Lie algebra level given by Proposition 4.5 of \cite{DDRW}.

\begin{proposition}  For $t = t_i + t_j + t_k \in \operatorname{Im}\mathbb{H} = \mathfrak{sp}(1),$ $$\phi_3(t) = \begin{bmatrix} 3t_i & \sqrt{3}(t_j + t_k)\\ \sqrt{3}(t_j + t_k) & 2(t_k - t_j) - t_i\end{bmatrix}$$ defines the unique irreducible $4$-dimensional representation of $\mathfrak{sp}(1) = \mathfrak{su}(2)$.

\end{proposition}

It follows that, for $H_1 = \{\diag(\phi_3(q_1), q_1): q_1 \in Sp(1) \}\subseteq Sp(3)$, $$\mathfrak{h}_1 = \left\{ \begin{bmatrix} 3t_i & \sqrt{3}(t_j + t_k) & \\ \sqrt{3}(t_j + t_k) & 2(t_k - t_j) - t_i & \\ & & t\end{bmatrix} : t\in \operatorname{Im}\mathbb{H}\right\}.$$  Likewise, for $H_2 = \{\diag(\phi_3(q_2),1): q_2\in Sp(1)\}\subseteq G$, we have $$\mathfrak{h}_2 = \left\{ \begin{bmatrix} 3s_i & \sqrt{3}(s_j + s_k) & \\ \sqrt{3}(s_j + s_k) & 2(s_k - s_j) - s_i & \\ & & 0\end{bmatrix} : s\in \operatorname{Im}\mathbb{H}\right\}.$$

\

The metric we will use is constructed in \cite{DDRW} via a combination of Cheeger deformations \cite{Ch1} and Wilking's doubling trick \cite{Wi}.  More specifically, we let $g_0$ denote the bi-invariant metric on $G$ with $g_0(X,Y) = -ReTr(XY)$ for $X,Y\in \mathfrak{g}$.  We let $g_1$ denote the left $G$-invariant, right $K$-invariant metric obtained by Cheeger deforming $g_0$ in the direction of $K$.  That is, $g_1$ is the metric induced on $G$ by declaring the canonical submersion $(G\times K, g_0 + g_0|_K)\rightarrow G$ with $(p,k)\mapsto p k^{-1}$ to be a Riemannian submersion.

We now equip $G\times G$ with the metric $g_1 + g_1$ and consider the isometric action of $G\times H_1\times H_2$ on $G\times G$ given by $(p,h_1,h_2)\ast(p_1, p_2) = (p\, p_1 \, h_1^{-1}, p \, p_2\, h_2^{-1})$.  This actions is free and induces a metric on the orbit space $\Delta G \backslash (G\times G)/ (H_1\times H_2)$.

Following Eschenburg \cite{Es2}, the orbit space $\Delta G\backslash (G\times G)/(H_1\times H_2)$ is canonically diffeomorphic to the biquotient $H_1\backslash G/H_2$, which is called $N_9$ in \cite{DDRW}.  To see this, one verifies the map $G\times G\rightarrow G$ sending $(p_1,p_2)$ to $p_1^{-1} p_2$ descends to a diffeomorphism of the orbit spaces.  We use this diffeomorphism to transport the submersion metric on $\Delta G\backslash (G\times G)/(H_1\times H_2)$ to $H_1\backslash G/ H_2$ and let $g_2$ to denote this metric on $H_1\backslash G/H_2$.

We note that since $g_0$ is bi-invariant, it is non-negatively curved.  It follows from O'Neill's formula \cite{On1} that $g_1$ and $g_2$ are non-negatively curved as well.

\

We now describe the points having $0$-curvature planes in $(H_1\backslash G/H_2, g_2)$.  To do this, we let $$\mathfrak{p} =\left\{\begin{bmatrix} 0 & 0 & z_1 \\ 0 & 0 &z_2 \\ -\overline{z}_1 & -\overline{z}_2 & 0\end{bmatrix}: z_1, z_2\in \mathbb{H}\right\}\subseteq \mathfrak{g}$$ denote the $g_0$-orthogonal complement of $\mathfrak{k}$: $\mathfrak{g} = \mathfrak{k}\oplus \mathfrak{p}$.  Then, for $X\in \mathfrak{g}$ we can write it as $X = X_{\mathfrak{k}} + X_{\mathfrak{p}}$ where $X_{\mathfrak{k}}$ is the projection of $X$ onto $\mathfrak{k}$, and similarly for $X_{\mathfrak{p}}$.  We also let $Ad_p:\mathfrak{g}\rightarrow \mathfrak{g}$ denote the adjoint map $Ad_p(X) = pXp^{-1}$.  Then, as shown in \cite{DDRW} (Corollary 2.8), we have the following description of points $[p^{-1}]\in H_1\backslash G/H_2$ containing $0$ curvature planes.

\begin{theorem}\label{preeqn}

There is a $0$-curvature plane at $[p^{-1}]\in (H_1\backslash G/H_2, g_2)$ iff there are linearly independent vectors $X,Y\in \mathfrak{g}$ satisfying the following equations.

(A)  $g_0(X,Ad_p \mathfrak{h}_1) = g_0(X,\mathfrak{h}_2) = g_0(Y,Ad_p\mathfrak{h}_1) = g_0(Y,\mathfrak{h}_2)=0$

(B)  $[X,Y] = [X_\mathfrak{k}, Y_\mathfrak{k}] = [X_\mathfrak{p}, Y_\mathfrak{p}] = 0$

(C)  $[(Ad_{p^{-1}} X)_\mathfrak{k}, (Ad_{p^{-1}} Y)_\mathfrak{k}] = [(Ad_{p^{-1}} X)_\mathfrak{p}, (Ad_{p^{-1}}Y)_\mathfrak{p}] = 0.$

\end{theorem}

It is clear from inspecting these equations that if $\operatorname{span}\{X,Y\} = \operatorname{span}\{X',Y'\}$, then $X$ and $Y$ satisfy all three conditions iff $X'$ and $Y'$ do.

We also note that there is some redundancy in these equations because $(G,K)$ is a symmetric pair.  Specifically, assuming $[X,Y]= 0$, it follows that $[X_\mathfrak{k},Y_\mathfrak{k}] = 0 $ iff $[X_\mathfrak{p}, Y_\mathfrak{p}] = 0$ and also that $[(Ad_{p^{-1}} X)_\mathfrak{k}, (Ad_{p^{-1}} Y)_\mathfrak{k}] = 0$ iff $[(Ad_{p^{-1}} X)_\mathfrak{p}, (Ad_{p^{-1}}Y)_\mathfrak{p}] = 0$.  To see this, we first note that $[\mathfrak{p},\mathfrak{p}]\subseteq \mathfrak{k}$ for a symmetric pair $(G,K)$.  Using the fact that $[\mathfrak{k},\mathfrak{p}]\subseteq \mathfrak{p}$, we see that $[X,Y]_\mathfrak{k} = [X_\mathfrak{k}, Y_\mathfrak{k}] + [X_\mathfrak{p}, Y_\mathfrak{p}]$.  Since condition 2 forces $[X,Y]_{\mathfrak{k}} = 0$, we see that $[X_\mathfrak{k},Y_\mathfrak{k}] = 0$ iff $[X_\mathfrak{p}, Y_\mathfrak{p}] =0$.  To get the result for the vectors $Ad_{p^{-1}} X$ and $Ad_{p^{-1}}Y$, we note that $Ad_{p^{-1}}:\mathfrak{g}\rightarrow \mathfrak{g}$ is Lie algebra isomorphism, so $[X,Y] = 0$ iff $[Ad_{p^{-1}} X,Ad_{p^{-1}} Y] = 0$.

We now show that for many $p\in Sp(3)$, if $X$ and $Y$ satisfy conditions (A) and (B) of Theorem \ref{preeqn}, then we may replace $X$ and $Y$ with $X',Y'$ having a nice form.

\begin{proposition}\label{niceform}  Let $\rho:\mathfrak{g}\rightarrow \mathbb{H}$ with $\rho\left(Z\right) = Z_{33}$, then entry of $Z$ in the last row and last column.  Suppose $[p^{-1}]\in G\bq H$ is a point for which $\rho|_{Ad_p \mathfrak{h_1}}$ is surjective.  If $X,Y\in \mathfrak{g}$ satisfy conditions (A) and (B) of Theorem \ref{preeqn} at the point $[p^{-1}]$, then there are vectors $X',Y'\in \mathfrak{g}$ with $\operatorname{span}\{X,Y\} = \operatorname{span}\{X',Y'\}$ and $X'_{\mathfrak{p}} = Y'_{\mathfrak{sp}(2)} = 0$, where $Y'_{\mathfrak{sp}(2)}$ denotes the projection of $Y'$ to $\mathfrak{sp}(2)\oplus 0 \subseteq \mathfrak{k}\subseteq \mathfrak{g}$.

\end{proposition}

\begin{proof}

We start with the equation $[X_{\mathfrak{p}}, Y_{\mathfrak{p}}] = 0$ from condition $(B)$.  Since we can identify $\mathfrak{p}$ with $T_{[eK]} G/K$, where $G/K = \mathbb{H}P^2$ has positive sectional curvature, it follows that $[X_{\mathfrak{p}}, Y_{\mathfrak{p}}] = 0$ iff $X_\mathfrak{p}$ and $Y_\mathfrak{p}$ are dependent.  Thus, either $X_\mathfrak{p} = 0$ and $X = X'$ or $X_\mathfrak{p} = \lambda Y_\mathfrak{p}$ for some real number $\lambda$.  Then $X' = \lambda X - Y$ has no $\mathfrak{p}$ part.  We may thus assume without loss of generality that $X$ has no $\mathfrak{p}$ part.

Since $Sp(2)\times\{I\}$ is an ideal in $K = Sp(2)\times Sp(1)$, the condition $[X_\mathfrak{k}, Y_\mathfrak{k}] = 0$ implies $[X_{\mathfrak{sp}(2)}, Y_{\mathfrak{sp}(2)}] = 0.$  By condition $(A)$, we know $g_0(X, \mathfrak{h}_2) = g_0(Y,\mathfrak{h_2}) = 0$, so we may interpret $X_{\mathfrak{sp}(2)}$ and $Y_{\mathfrak{sp}(2)}$ as tangent vectors on $Sp(2)/\phi_3(Sp(1))$.  But, $Sp(2)/\phi_3(Sp(1))$ is the Berger space \cite{Ber} and is known to admit a normal homogeneous metric of positive curvature.  So we see that $[X_{\mathfrak{sp}(2)}, Y_{\mathfrak{sp}(2)}] = 0$ iff $X_{\mathfrak{sp}(2)}$ and $Y_{\mathfrak{sp}(2)}$ are linearly dependent.  

If $X_{\mathfrak{sp}(2)} = 0$, then the only non-vanishing entry of $X$ is $X_{33}$.  Since, by assumption, $\rho|_{Ad_p \mathfrak{h_1}}$ is surjective, the condition $g_0(X, Ad_p \mathfrak{h_1}) = 0$ forces $X = 0$, contradicting the fact that $\{X,Y\}$ is linearly independent.  Hence, we may assume $X_{\mathfrak{sp}(2)} \neq 0$.  Then, we may subtract an appropriate multiple of $X$ from $Y$ to obtain a new vector $Y'$ with $Y'_{\mathfrak{sp}(2)} = 0$.

\end{proof}

We now work out conditions (A), (B), and (C) of Theorem \ref{preeqn} more explicitly at a point of the form $p = \begin{bmatrix} \cos\theta & 0& \sin\theta \\0 & 1 & 0\\ -\sin\theta & 0& \cos\theta\end{bmatrix}$.  We will always assume $\theta \in \left(0,\pi/4\right)$.  Also, we will often identify $\mathfrak{p}$, consisting of matrices of the form $\begin{bmatrix} 0 & 0 & z_1\\ 0 & 0 & z_2\\ -\overline{z}_1 & -\overline{z}_2 & 0\end{bmatrix}$, with $\mathbb{H}^2$ via the canonical $\mathbb{R}$-linear isomorphism mapping such a matrix to $\begin{bmatrix} z_1 \\ z_2\end{bmatrix}$.

We note that for points of this form, $\rho|_{Ad_p \mathfrak{h_1}}$ has image consisting of all elements of $\operatorname{Im}\mathbb{H}$ of the form $3\sin^2\theta t_i + \cos^2\theta t$ for $t = t_i + t_j + t_k\in\operatorname{Im}\mathbb{H}$.  Since $\cos^2\theta \neq 0$ because $\theta\in(0,\pi/4)$, this map has no kernel, so is surjective.  In particular, the conditions of Proposition \ref{niceform} are verified at all such $p$, and thus, we may assume $X = \begin{bmatrix} x_1 & x_2 & 0 \\ -\overline{x}_2 & x_3 & 0 \\0 & 0& x_4\end{bmatrix}$ with $x_1,x_3, x_3\in \operatorname{Im}\mathbb{H}$ and $x_2\in \mathbb{H}$.  Similarly, we may assume $Y = \begin{bmatrix} 0 & 0 & y_1\\ 0 & 0 & y_2 \\ -\overline{y}_1 & -\overline{y}_2 & y_3\end{bmatrix}$ with $y_1,y_2\in \mathbb{H}$ and $y_3\in \operatorname{Im}\mathbb{H}$

\begin{lemma}\label{equations}  For a point $p$ of the above form and $X,Y\in \mathfrak{g}$, conditions (A), (B), and (C) of Theorem \ref{preeqn} are equivalent to the following list of conditions.

\begin{equation}\label{eqn1} x_1 y_1 + x_2 y_2 - y_1 x_4 = 0 \tag{1}\end{equation}

\begin{equation}\label{eqn2} -\overline{x_2}y_1 + x_3 y_2 - y_2 x_4 = 0 \tag{2} \end{equation}

\begin{equation} \label{eqn3} \{x_4,y_3\} \text{ is linearly dependent over } \mathbb{R} \tag{3}\end{equation}

For $$v = \begin{bmatrix} \cos\theta\sin\theta(x_1 - x_4)\\ -\sin\theta\overline{x_2}\end{bmatrix} \in\mathbb{H}^2\cong \mathfrak{p}$$ and $$w = \begin{bmatrix} \operatorname{Re}(y_1) + (\cos^2\theta - \sin^2\theta) \operatorname{Im}(y_1)- \sin\theta\cos\theta\, y_3\\ \cos\theta y_2\end{bmatrix}\in \mathbb{H}^2\cong \mathfrak{p},$$ \begin{equation}\label{eqn4} \text{the set }\{v,w\} \text{ is linearly dependent over } \mathbb{R}  \tag{4}\end{equation}

\begin{equation} \label{eqn5i} 3(x_1)_i - (x_3)_i = 0 \tag{5i}\end{equation}

\begin{equation} \label{eqn5j} \sqrt{3}(x_2)_j - (x_3)_j = 0 \tag{5j}\end{equation}

\begin{equation} \label{eqn5k} \sqrt{3}(x_2)_k + (x_3)_k = 0 \tag{5k}\end{equation}

\begin{equation} \label{eqn6i} (x_1)_i (-2\sin^2\theta ) + (x_4)_i(1+2\sin^2\theta ) = 0\tag{6i} \end{equation}

\begin{equation} \label{eqn6j} (x_2)_j (\cos\theta-1)2\sqrt{3} + (x_1)_j \sin^2\theta + (x_4)_j \cos^2\theta = 0 \tag{6j} \end{equation}

\begin{equation} \label{eqn6k} (x_2)_k (\cos\theta-1)2\sqrt{3} +  (x_1)_k \sin^2\theta +(x_4)_k \cos^2\theta = 0 \tag{6k} \end{equation}

\begin{equation} \label{eqn7i} -4\sin\theta\cos\theta (y_1)_i +(2\sin^2\theta + 1)(y_3)_i = 0 \tag{7i} \end{equation}

\begin{equation} \label{eqn7j} 2\sin\theta\cos\theta(y_1)_j -2\sqrt{3}\sin\theta (y_2)_j + \cos^2\theta(y_3)_j = 0 \tag{7j} \end{equation}

\begin{equation} \label{eqn7k} 2\sin\theta\cos\theta(y_1)_k -2\sqrt{3}\sin\theta (y_2)_k + \cos^2\theta(y_3)_k = 0 \tag{7k} \end{equation}

\end{lemma}

\begin{proof}

We first claim that condition (A) is equivalent to equations \eqref{eqn5i} through \eqref{eqn7k}.  To begin with, we note that since $Y_{\mathfrak{sp}(2)} = 0$ and $\mathfrak{h}_2\subseteq \mathfrak{sp}(2)\oplus 0\subseteq \mathfrak{k}$, the equation $g_0(Y,\mathfrak{h}_2) = 0$ is automatically satisfied.

Now, a calculation shows that for $s = s_i + s_j + s_k\in \operatorname{Im}\mathbb{H}$, $$0=g_0(X,\mathfrak{h}_2) = 3s_i x_1 +2\sqrt{3}(s_j+s_k) \operatorname{Im}(x_2) + (2(s_k-s_j) - s_i) x_3.$$  Then, using each of $s = i$, $s=j$, and $s=k$ gives equations \eqref{eqn5i}, \eqref{eqn5j}, \eqref{eqn5k} which, using linearity, are therefore equivalent to the condition that $g_0(X,\mathfrak{h}_2) = 0$.  

Further with $t = t_i + t_j + t_k \in \operatorname{Im}\mathbb{H}$ we compute $$Ad_p \mathfrak{h}_1 = \left\{\begin{bmatrix} 3\cos^2\theta t_i + \sin^2\theta t &  \sqrt{3}\cos\theta(t_j + t_k) & \cos\theta \sin\theta (t - 3t_i)\\ \sqrt{3}\cos\theta (t_j + t_k) & 2(t_k - t_j) - t_i & -\sqrt{3}\sin\theta(t_j + t_k)\\ \cos\theta \sin\theta (t - 3t_i) & -\sqrt{3}\sin\theta(t_j + t_k) & 3\sin^2\theta t_i + \cos^2\theta t  \end{bmatrix} \right\}.$$  A calculation now shows that the expression $g_0(X,Ad_p \mathfrak{h}_1)$ is given by the expression \begin{align*} &(3\cos^2 \theta t_i + \sin^2\theta t) x_1 +2\sqrt{3} \cos\theta (t_j + t_k) \operatorname{Im}(x_2)\\ & + (2(t_k-t_j) - t_i)x_3 + (3\sin^2\theta t_i + \cos^2\theta t)x_4 .\end{align*} Substituting each of $t =i$, $t=j$, and $t=k$ and using \eqref{eqn5i}, \eqref{eqn5j}, and \eqref{eqn5k} to eliminate $x_3$ gives, after using $\sin^2\theta + \cos^2\theta = 1$, \eqref{eqn6i}, \eqref{eqn6j}, \eqref{eqn6k}.

Likewise, the equation $g_0(Y, Ad_p \mathfrak{h}_1) = 0$ is equivalent to the vanishing of the expression $$2\cos\theta\sin\theta(-3t_i + t)\operatorname{Im}(y_1) -2\sqrt{3}\sin\theta (t_j+t_k)\operatorname{Im}(y_2) + (3\sin^2\theta t_i + \cos^2\theta t) y_3 .$$  Substituting each of $t =i$, $t=j$, and $t=k$ gives equations \eqref{eqn7i}, \eqref{eqn7j}, and \eqref{eqn7k}.

\

We next claim that equations \eqref{eqn1}, \eqref{eqn2}, and condition \eqref{eqn3} are equivalent to condition (B) of Theorem \ref{preeqn}.  Computing, we see $ [X,Y] = 0$ iff equations \eqref{eqn1} and \eqref{eqn2} are satisfied and $[x_4,y_3] = 0$.  But this latter condition is equivalent to \eqref{eqn3} since $Sp(1) = S^3$ has positive sectional curvature.  Further $X_{\mathfrak{p}} = 0$ so $[X_\mathfrak{p}, Y_\mathfrak{p}] = 0$ and since $Y_{\mathfrak{sp}(2)} = 0$, $[X_{\mathfrak{k}},Y_\mathfrak{k}] = 0$ iff condition \eqref{eqn3} is satisfied.

Lastly, we claim that condition \eqref{eqn4} is equivalent to condition (C) of Theorem \ref{preeqn}.  To see this, first recall that it was shown directly following Theorem \ref{preeqn} that the conditions $[(Ad_{p^{-1}} X)_\mathfrak{k}, (Ad_{p^{-1}} Y)_\mathfrak{k}] = 0$ and $ [(Ad_{p^{-1}} X)_\mathfrak{p}, (Ad_{p^{-1}}Y)_\mathfrak{p}] = 0$ are equivalent, so we may focus on only one of these.

A direct calculation shows that $v = (Ad_{p^{-1}} X)_\mathfrak{p}$ and $w = (Ad_{p^{-1}}Y)_\mathfrak{p}$, so we need only argue that $[v,w] = 0$ iff $v$ and $w$ are dependent over $\mathbb{R}$.  But we may interpret $v,w$ as elements of $T_{[eK]}G/K$ where $G/K = \mathbb{H}P^2$ has a normal bi-invariant metric of positive sectional curvature.  It follows that the bracket of $v$ and $w$ vanishes iff $v$ and $w$ are linearly dependent.
\end{proof}

\section{Quasi-positive curvature}

In this section, we prove $N_9 = H_1\backslash Sp(3)/H_2$ is quasi-positively curved with the metric $g_2$ constructed in Section \ref{background}.  As mentioned above, the metric $g_2$ is non-negatively curved, so it is sufficient to find a single point for which all $2$-planes have non-zero curvature.  In fact, we will show the following theorem.

\begin{theorem}\label{precise} With respect to the metric $g_2$, $N_9$ is positively curved at points of the form $[p^{-1}]\in H_1\backslash G/H_2\cong N_9$, where $p = \begin{bmatrix} \cos\theta & & -\sin\theta\\ 0 & 1 & 0 \\ \sin\theta & 0 & \cos\theta\end{bmatrix}$ with $\theta \in(0,\pi/6)$.

\end{theorem}

We will always work with points $p$ of the above form.

Assume $[p^{-1}]\in H_1\backslash G/ H_2$ is a point for which there is a $0$-curvature plane.  Then, using Theorem \ref{preeqn} and Proposition \ref{niceform}, it follows that there are linearly independent $X,Y,\in\mathfrak{g} = \mathfrak{sp}(3)$ with $X = \begin{bmatrix} x_1 & x_2 & 0 \\ -\overline{x}_2 & x_3 & 0\\ 0 & 0& x_4\end{bmatrix}$ and $Y = \begin{bmatrix} 0 & 0 & y_1\\ 0 & 0 & y_2\\ -\overline{y}_1 & -\overline{y_2} & y_3\end{bmatrix}$ and which satisfy all of the conditions given by Lemma \ref{equations}.  By repeatedly applying the conditions of Lemma \ref{equations}, we will constrain the forms of $X$ and $Y$ until we finally find that no such $X$ and $Y$ exist.  This contradiction will establish that there are no zero curvature planes at $[p]^{-1}$, and hence, that $N_9$ is positively curved at these points.

\begin{proposition}

If $\theta\in(0, \pi/6)$, the two vectors $$v = \begin{bmatrix} \cos\theta\sin\theta (x_1 - x_4)\\ -\sin\theta\overline{x_2}\end{bmatrix} $$ and $$ w = \begin{bmatrix} \operatorname{Re}(y_1) + (\cos^2\theta - \sin^2\theta) \operatorname{Im}(y_1)- \sin\theta\cos\theta\, y_3\\ \cos\theta y_2\end{bmatrix} $$ are both non-zero.
\end{proposition}

\begin{proof}

Suppose for a contradiction that $v = 0$.  Since $0 < \theta < \pi/6$, $v = 0$ implies $x_2 = 0$ and $x_1 = x_4$.  Then equations \eqref{eqn6i}, \eqref{eqn6j}, and \eqref{eqn6k} imply $x_1 = x_4 = 0$.  Then equations \eqref{eqn5i}, \eqref{eqn5j}, and \eqref{eqn5k} imply $x_3$ vanishes as well.  Thus, in this case, $X = 0$, contradicting the fact that $X$ and $Y$ are linearly independent.  Thus, $v\neq 0$.

\

Now, suppose $w = 0$, so $y_2 = 0$, $\operatorname{Re}(y_1) = 0$ and $\operatorname{Im}(y_1) = y_1 = \frac{\sin\theta \cos\theta}{\cos^2\theta-\sin^2\theta} y_3$.  The latter equation implies that the $i$,$j$, and $k$ components of $y_1$ and $y_3$ are positive multiplies of each other.  However,  equations \eqref{eqn7j} and \eqref{eqn7k} imply the $j$ and $k$ components of $y_1$ and $y_3$ are negative multiples of each other.  Thus, we must have $(y_1)_j = (y_1)_k = (y_3)_j = (y_3)_k = 0$.

Rearranging equation \eqref{eqn7i}, we see $y_1 = \frac{2\sin^2\theta + 1}{4\sin\theta\cos\theta} y_3$.  Combining this with the above equation $y_1 = \frac{\sin\theta \cos\theta}{\cos^2\theta - \sin^2\theta} y_3$, we see that either $y_1 = y_3 = 0$, or $\theta$ must satisfy the equation $$\frac{\sin\theta \cos\theta}{\cos^2\theta - \sin^2\theta} = \frac{2\sin^2\theta + 1}{4\sin\theta\cos\theta}.$$  Clearing denominators and simplifying gives $2\sin^2\theta \cos^2\theta + 2\sin^4 \theta +\sin^2\theta = \cos^2\theta $.  Factoring $\sin^2\theta$ out of the expression $2\sin^2 \theta \cos^2\theta + 2\sin^4\theta$, we see this expression simplifies to $2\sin^2\theta$.  Substituting this back in gives the equation $3\sin^2\theta = \cos^2\theta$ which has no solutions in $(0,\pi/6)$.

Thus, for $\theta \in (0,\pi/6)$, we conclude $y_1 = y_3 = 0$, which implies $Y = 0$, again contradicting the fact that $X$ and $Y$ are linearly independent.

\end{proof}

Using condition \eqref{eqn4}, it follows that by rescaling $X$, we may thus assume $v = w$.  Further, the first component of $v$ is purely imaginary, and hence $\operatorname{Re}(y_1) = 0 $, that is, $y_1 = \operatorname{Im} y_1$.  Thus, the condition \eqref{eqn4} is equivalent to the following two equations:  \begin{equation}\label{eqnv1} \cos\theta\sin\theta\, (x_1 - x_4) = (\cos^2\theta - \sin^2\theta)\, y_1 - \sin\theta \cos\theta\, y_3 \tag{4.1}\end{equation}  \begin{equation}\label{eqnv2} y_2 = -\tan\theta \,\overline{x}_2 . \tag{4.2}\end{equation}

\begin{proposition}\label{fnot0}  For any $\theta\in (0,\pi/6)$, $x_2, y_1,$ and $y_2$ are all non-zero.  \end{proposition}

\begin{proof}

Assume for a contradiction that $y_2 = 0$.    Note that, because all the coefficients in equations \eqref{eqn7i}, \eqref{eqn7j}, \eqref{eqn7k} are non-zero, it follows that $y_1 = 0$ iff $y_3 = 0$.  Because $Y\neq 0$, it follows that $y_1\neq 0$.

Rearranging \eqref{eqn1} gives $x_1 y_1 = y_1 x_4$.  Taking lengths, we see that $|x_1| = |x_4|$.  We now compare the $i$, $j$, and $k$ component of $x_1$ and $x_4$.

For the $i$ component, we rearrange equation \eqref{eqn6i} to obtain $$(x_1)_i = \frac{1+2\sin^2\theta}{2\sin^2\theta} (x_4)_i = \left(1 + \frac{1}{2\sin^2\theta }\right) (x_4)_i.$$  Since $1 + \frac{1}{2\sin^2\theta} > 0$, we conclude that $|(x_1)_i| \geq |(x_4)|_i$ with equality iff $(x_1)_i = (x_4)_i = 0$.

For the $j$ component, we first remark that equation \eqref{eqnv2} shows that $x_2=0$ because $y_2 = 0$.  Then, rearranging equations \eqref{eqn6j} gives $$(x_1)_j = -\frac{\cos^2\theta }{\sin^2\theta} (x_4)_j .$$ Thus, since $0<\theta < \pi/6$, we conclude that $|(x_1)_j|\geq |(x_4)_j|$ with equality iff $(x_1)_j = (x_4)_j = 0$.  The same argument shows $|(x_1)_k|\geq |(x_4)_k|$ with equality iff $(x_1)_k = (x_4)_k = 0.$

Thus, each component of $x_1$ is at least as large, in magnitude, as the corresponding component of $x_4$.  Hence, since $|x_1| = |x_4|$, it follows that each of these inequalities must be equalities, so $x_1 = x_4 = 0$.  Since we have already shown $x_2 = 0$, equations \eqref{eqn5i}, \eqref{eqn5j}, and \eqref{eqn5k} force $x_3 = 0$ as well.  That is, $X = 0$, a contradiction.  Thus, $y_2\neq 0$.

Finally, it follows from equation \eqref{eqnv2} that $x_2\neq 0$, and from \eqref{eqn1}, we see that since $x_2y_2\neq 0$, that $y_1\neq 0$.
\end{proof}

We now show that $x_1$ cannot be equal to $x_4$.

\begin{proposition}\label{anotd} For every $\theta \in (0,\pi/6)$, $x_1\neq x_4$.

\end{proposition}

\begin{proof} Suppose for a contradiction that $x_1 = x_4$.  Then equation \eqref{eqn1} takes the form \begin{align*} 0 &= x_1 y_1 - y_1 x_1 -\tan\theta |x_2|^2 \\ &= [x_1,y_1] - \tan\theta |x_2|^2.\end{align*}  Since $x_1,y_1 \in\operatorname{Im}\mathbb{H}$, $[x_1,y_1]\in\operatorname{Im}\mathbb{H}$ as well, so we conclude that $\tan\theta |x_2|^2 = 0$.  Since $0<\theta<\pi/6$, it follows that $x_2 = 0$, a contradiction.

\end{proof}

Our next goal is to demonstrate the following proposition.

\begin{proposition}\label{dependent}

For every $\theta\in (0,\pi/6)$, $\dim_{\mathbb{R}} \operatorname{span}_{\mathbb{R}}\{x_1, x_4, y_1,y_3\} = 1$.
\end{proposition}

Of course, since, by Proposition \ref{fnot0}, $y_1\neq 0$, the dimension of this span is at least $1$, so we need only show it is at most one.  We first show Proposition \ref{dependent} holds when $x_4 = 0$.

\begin{proposition}\label{x4not0} Assume $x_4 = 0$.  Then, for every $\theta \in (0,\pi/6)$, $$\dim_{\mathbb{R}} \operatorname{span}_{\mathbb{R}}\{x_1, x_4, y_1,y_3\} = 1.$$

\end{proposition}

\begin{proof}
Equation \eqref{eqn1} with $x_4 = 0$ is $x_1 y_1 - \tan\theta |x_2|^2 = 0$.  In particular, $x_1 y_1\in\mathbb{R}$.  Since $x_1$ and $y_1$ are purely imaginary, this implies $\{x_1,y_1\}$ is linearly dependent over $\mathbb{R}$.  Now \eqref{eqnv2} implies that $y_3 =  -x_1 + \frac{\cos^2\theta - \sin^2\theta}{\cos\theta\sin\theta} y_1$, so $\{x_1,y_1, y_3\}$ is linearly dependent.  Since $x_4 = 0$, $\operatorname{span}_\mathbb{R} \{x_1, x_4, y_1, y_3\}$ is $1$ dimensional.

\end{proof}

We now investigate the case where $x_4\neq 0$.  Then by condition \eqref{eqn3}, we may write $y_3 = \lambda x_4$ for some real number $\lambda$.

\begin{proposition}\label{choice}  For $\theta \in (0,\pi/6)$, either $\lambda = 2$ or $\dim_\mathbb{R} \operatorname{span}_\mathbb{R}\{x_1, x_4, y_1, y_3\} = 1$.

\end{proposition}

\begin{proof}

Rearranging \eqref{eqnv1} gives $y_1 = \cos\theta \sin\theta (x_1 + (\lambda - 1) x_4)$.  Substituting this into \eqref{eqn1} gives \begin{equation}\label{lambdasimp} 0 = \cos\theta \sin\theta\left[ x_1 (x_1 + (\lambda-1)x_4) - (x_1 + (\lambda -1)x_4)x_4\right] -\tan\theta |x_2|^2.\tag{$\ast_1$}\end{equation}  Recalling the square of a purely imaginary number is real, the imaginary part of equation \eqref{lambdasimp} simplifies to $$0 = \cos\theta\sin\theta (\lambda - 2) \operatorname{Im}(x_1 x_4).$$

So, if $\lambda \neq 2$, then $\operatorname{Im}(x_1x_4) = 0$, that is, $\{x_1, x_4\}$ must be linearly dependent.  Recalling $y_3 = \lambda x_4$ and $y_1 = \cos\theta \sin\theta(x_1 + (\lambda -1) x_4)$, we see that if $\lambda \neq 2$, then $\dim_\mathbb{R} \operatorname{span}_\mathbb{R} \{x_1, x_4, y_3, y_1\} = 1$.

\end{proof}

We now show $\lambda = 2$ cannot occur.  Once we show this, Proposition \ref{choice} will then imply Proposition \ref{dependent} holds even when $x_4\neq 0$.

\begin{proposition}  For $\theta \in (0,\pi/6)$, $\lambda \neq 2$.

\end{proposition}

\begin{proof}

Assume for a contradiction that $\lambda = 2$.  We first show this implies the $j$ and $k$ components of $x_2$ and $y_2$ must vanish.

Given $x_2$ and $x_4$, equations \eqref{eqn6j} and \eqref{eqn6k} determine $x_1$:  $$x_1 = -\frac{\cos^2\theta \, x_4 + 2\sqrt{3}(\cos\theta - 1) x_2}{\sin^2 \theta}.$$  Substituting this into equation \eqref{eqnv1} and rearranging gives $$y_1 =  -\frac{\cos\theta}{\sin\theta} x_4 -\frac{2\sqrt{3} \cos\theta(\cos\theta - 1)}{\sin\theta(\cos^2\theta - \sin^2\theta)} x_2.$$  Then substituting this into \eqref{eqn7j} and \eqref{eqn7k}, we determine $(y_2)_j$ and $(y_2)_k$, $$(y_2)_j = -\frac{-2\cos^2\theta (\cos\theta - 1)}{\sin\theta (\cos^2\theta - \sin^2\theta)}(x_2)_j$$ and likewise for the $k$ component.

On the other hand, from equation \eqref{eqnv2}, $y_2 = -\frac{s}{c} \overline{x_2}$, so the $j$ and $k$ parts of $y_2$ are determined in a different way by those of $x_2$.  Thus, either $(x_2)_j = (y_2)_j = (x_2)_k = (y_2)_k = 0$ or \begin{equation}\label{lambdanot2} -\frac{-2\cos^2\theta (\cos\theta - 1)}{\sin\theta (\cos^2\theta - \sin^2\theta)} = \frac{\sin\theta}{\cos\theta}.\tag{$\ast_2$}\end{equation}  By clearing denominators and replacing $\sin^2 x$ with $1-\cos^2 x$ everywhere, \eqref{lambdanot2} is equivalent to $2\cos^3\theta - 3\cos^2\theta + 1 = 0$, which factors as $(\cos \theta - 1)^2(2\cos\theta + 1) = 0$.  But, on $(0,\pi/6)$, $0 < \cos\theta < 1$, so this equation is never satisfied on $(0,\pi/6)$.  It follows that if $\lambda =2$, then the $j$ and $k$ components of $x_2$ and $y_2$ vanish.

Because the $j$ and $k$ components of $x_2$ vanish, the proof of Proposition \ref{fnot0} shows that $|x_1|\geq |x_4|$ with equality only if $|x_1| = |x_4| = 0$.

Now, \eqref{eqnv1} gives $y_1 = \frac{\cos\theta \sin\theta}{cos^2\theta-\sin^2\theta}(x_1 + x_4)$.  Substituting this this into \eqref{eqn1}, we get $$\frac{\cos\theta \sin\theta}{\cos^2\theta -\sin^2\theta}(x_1^2 - x_4^2) =  \frac{\sin\theta}{\cos\theta}|x_2|^2.$$  Because $x_1$ is purely imaginary, $x_1^2 = -|x_1|^2$ and similarly for $x_4$, so this equation is equivalent to \begin{equation}\label{lambdanot0again} \frac{\cos\theta \sin\theta}{\cos^2\theta-\sin^2\theta}(|x_4|^2 - |x_1|^2) = \frac{\sin\theta}{\cos\theta} |x_2|^2.\tag{$\ast_3$}\end{equation}  Note that for $\theta \in (0,\pi/4)$, the coefficients in \eqref{lambdanot0again} on both the left hand side and right hand side are positive, and so, by Proposition \ref{fnot0}, the right side is positive.

On the other hand, since $|x_1|\geq |x_4|$, the left side is non-positive.  This contradiction implies $\lambda = 2$ cannot occur for any $\theta \in (0,\pi/6)$.

\end{proof}

Combining this with Proposition \ref{choice}, we may assume $\dim_{\mathbb{R}}\operatorname{span}_\mathbb{R}\{x_1, x_4, y_1, y_3\} = 1$.  In particular, the quaternions $x_1, x_4, y_1$, and $y_3$ commute with each other.  Further, since we have already shown $y_1\neq 0$,  each of $x_1$, $x_4$, and $y_3$ can be written as a multiple of $y_1$.

\begin{proposition}\label{noi}  Suppose $\theta \in (0,\pi/6)$.  Then the $i$ components of $x_1,x_4,y_1$, $y_3$ and $x_3$ are all zero.

\end{proposition}

\begin{proof}

If $(y_1)_i = 0$, it follows from Proposition \ref{dependent}, together with the fact that $y_1\neq 0$ (Proposition \ref{fnot0}), that the $i$ component of $x_1$, $x_4$, and $y_3$ are all $0$ as well.  Then equation \eqref{eqn5i} shows $(x_3)_i = 0$ as well.  So, we need only show $(y_1)_i = 0$ when $\theta\in (0,\pi/6)$.

From equation \eqref{eqn7i}, using the fact that $y_3$ is a multiple of $y_1$, we see $y_3 = \frac{4\sin\theta\cos\theta}{2\sin^2\theta + 1} y_1$.  Substituting this into equation \eqref{eqnv1}, we see $\cos\theta\sin\theta(x_1 - x_4) = \cos^2\theta - \sin^2\theta - \cos\theta \sin\theta \frac{4\cos\theta\sin\theta}{2\sin^2\theta + 1} y_1$.  Upon substituting $\cos^2\theta = 1-\sin^2\theta$, the coefficient on the right simplifies to $\frac{1-4\sin^2 \theta }{1 + 2\sin^2 \theta}$.  Since $\theta \in(0,\pi/6)$, this coefficient is positive.  It follows that $x_1 - x_4$ is a positive multiple of $y_1$.

Now, note that equation \eqref{eqn1}, rearranged, takes the form $(x_1 - x_4) y_1 = \tan\theta |x_2|^2$.  Since $\theta\in(0,\pi/6)$,  the right hand side is positive.  But since $x_1 - x_4$ is a positive multiple of $y_1$, the left hand side is a positive multiple of $y_1^2$.  The square of any purely imaginary number is non-positive, so we have a contradiction.

\end{proof}

We now show that $x_3$ must be non-zero.

\begin{proposition}\label{x3not0} $x_3 \neq 0$

\end{proposition}

\begin{proof}

Suppose for a contradicting that $x_3 = 0$.  By equations \eqref{eqn5j} and \eqref{eqn5k}, $x_2$ has no $j$ or $k$ component.  Since $y_2 = -\tan\theta \,\overline{x}_2$, the $j$ and $k$ components of $y_2$ vanish as well.

Now, equations \eqref{eqn7j} and \eqref{eqn7k} give $y_3 = -2\tan\theta\, y_1$.  In particular, $y_3$ is a negative multiple of $y_1$.  From equation \eqref{eqnv1}, we now see $\cos\theta \sin \theta (x_1 -x_4)$ is a positive multiple of $y_1$.  Then, just as in the proof of Proposition \ref{noi}, this contradicts equation \eqref{eqn1}.

\end{proof}

We also find the $j$ and $k$ components of $x_2$ and $y_2$ are constrained.

\begin{proposition}\label{jkparallel} Let $x_2'$, $y_2'$ denote the projection of $x_2$ and $y_2$ into the $jk$-plane.  Then $\dim_\mathbb{R} \operatorname{span}_\mathbb{R}\{x_1, x_4, y_1, y_3 x_2',y_2'\} = 1$.
\end{proposition}

\begin{proof}

Recalling that $x_1$ and $x_4$ have no $i$ component by Proposition \ref{no}, we see that multiplying equation \eqref{eqn6j} by $j$ and equation \eqref{eqn6k} by $k$ and adding gives the equation $$(x_2')(\cos\theta - 1)2\sqrt{3} + x_1 \sin^2\theta + x_4 \cos^2\theta = 0.$$  Thus, $x_2'$ is dependent on $x_1$ and $x_4$.  Since $y_2 = -\tan\theta \,\overline{x}_2$, $y_2' = \tan\theta\, x_2'$ so is dependent on $x_2'$.  The result follows.

\end{proof}

\begin{proposition}\label{jk0}  Either $(x_2)_j = 0$ or $(x_2)_k = 0$, but not both.

\end{proposition}

\begin{proof}  If both are zero, then equations \eqref{eqn5j} and \eqref{eqn5k} give $x_3 = 0$, contradicting Proposition \ref{x3not0}.  We now show at least one vanishes.

We begin by rearranging \eqref{eqn2} into the form $$\overline{x_2}(\tan\theta \,x_4-y_1) = \tan\theta x_3 \overline{x_2}.$$  We write $x_2 = x_2'' + x_2 '$ as a decomposition into the complex components, together with the $j$ and $k$ components.  That is, $x_2''\in \mathbb{C}$ while $x_2'\in\operatorname{span}\{j,k\}$, as before.  Then, the left hand side can be expanded as $x_2''(\tan \theta\, x_4 - y_1) + x_2' (\tan \theta \, x_4 - y_1)$.  Recalling that the $i$ component of $x_4$ and $y_1$ vanishes by Proposition \ref{noi}, $x_2'' (\tan \theta\, x_4 - y_1)\in \operatorname{span}\{j,k\}$.

Further, because $x_2'$ is dependent on both $x_4$ and $y_1$ by Proposition \ref{jkparallel}, $x_2'(\tan\theta x_4 - y_1)\in \mathbb{R}$.  It follows that the left hand side, that is, $x_2(\tan\theta\, x_4 -y_1)$ has no $i$ component.

Hence, the $i$ component of the right hand side, $\tan\theta \, x_3 \overline{x}_2$, must vanish as well.  Since $(x_3)_i = 0$ by Proposition \ref{noi}, the $i$ component of $x_3 \overline{x}_2$ is given by \begin{align*} 0&=  (x_3 \overline{x}_2)_i i\\ &= (x_3)_j j (\overline{x}_2)_k k + (x_3)_k k (\overline{x}_2)_j j\\ &= (-(x_3)_j (x_2)_k + (x_3)_k (x_2)_j)i.\end{align*}  Now, using equations \eqref{eqn5j} and \eqref{eqn5k}, we see that $(x_3)_j = \sqrt{3}(x_2)_j$ and $(x_3)_k = -\sqrt{3}(x_2)_k$.  Substituting this in yields $0 = -2\sqrt{3}(x_2)_j (x_2)_k$, so at least one of $(x_2)_j$ and $(x_2)_k$ vanishes.

\end{proof}

Since we have already shown $\dim_{\mathbb{R}} \operatorname{span}\{x_1, x_4, y_1, y_3, x_2', y_2'\} = 1$ (Proposition \ref{jkparallel}), it follows that, either they all only have a $k$ component, or they all only have a $j$ component.  Equations \eqref{eqn5j} and \eqref{eqn5k} now show that $x_3$ is also in the span of $\{x_1, x_4 , y_y1, _3, x_2', y_2'\}$.

Our next proposition will show that all the variables must commute.

\begin{proposition}  $(x_2)_i = (y_2)_i= 0$.

\end{proposition}

\begin{proof}  Since $y_2 = -\tan\theta \overline{x}_2$, it is enough to show that $(x_2)_i = 0$.

Equation \eqref{eqn2} can be rearranged into the form $$\tan \theta\, x_4 - y_1 = \frac{\tan\theta}{|x_2|^2} x_2 x_3 \overline{x}_2.$$  By Propositions \ref{noi}, \ref{jkparallel}, and \ref{jk0}, the left hand side, $x_3$, and $x_2$' are all either a real multiple of $j$ or a real multiple of $k$.  For the remainder of the proof, we assume they are all multiples of $j$; the case where they are multiples of $k$ is identical.

The right side is, up to multiple, given by conjugating $x_3$ by the unit quaternion $\frac{x_2}{|x_2|}$.  Recall that a unit quaternion can be written as $q = \cos\theta q_0 + \sin\theta q_1$ where $q_0$ is real and $q_1$ is purely imaginary and $|q_0| = |q_1| = 1$.  Then conjugation by $q$, viewed as a map from $\mathbb{R}^3\cong \operatorname{Im}(\mathbb{H})$ to itself, is a rotation with axis given by $q_1$ and with rotation angle given by $2\theta$.

Since the $j$-axis is invariant under conjugation by $x_2$, we see one of two things happen.  Either the $j$-axis is fixed point wise, in which case $\operatorname{Im}(x_2)$ has only a $j$ component, or the orientation of it is reversed.  We now show the latter case cannot occur.

If the orientation is reversed, the the rotation axis $\operatorname{Im}(x_2)$ must be perpendicular to $j$, so $\operatorname{Im}(x_2)\in \operatorname{span}\{i,k\}$.  Because $x_2'$ has no $k$ part, so it follows that $x_2' = 0$.  But then, using equations \eqref{eqn5j} and \eqref{eqn6j}, we see that $x_3 = 0$, contradicting Proposition \ref{x3not0}.

\end{proof}

It follows that $\operatorname{Im}(x_2) = x_2'$.  Summarizing, we have now shown that at a point containing a $0$ curvature plane with $\theta \in (0,\pi/6)$, that $x_2 ' = \operatorname{Im}(x_2)$, $y_2' = \operatorname{Im}(y_2)$, that $\dim_\mathbb{R} \operatorname{span}\{x_1,x_3, x_4, y_1, y_3, x_2', y_2'\} = 1$ and further, that each element in this set has vanishing $i$ and $j$ components or vanishing $i$ and $k$ components.  In particular, the variables $x_1, x_2, x_3, x_4, y_1, y_2,$ and $y_3$ all commute.  Thus, after substituting $y = -\tan \theta\, \overline{x}_2$ into equation \eqref{eqn2} and then canceling all occurrences of $\overline{x}_2$, we may replace equation \eqref{eqn2} with the linear equation \begin{equation*}\tan\theta\, x_4 - \tan\theta \, x_3 -y_1 = 0.\end{equation*}

We let $\ell\in\{j,k\}$ and set $\epsilon = 1$ if $\ell = j$ and $\epsilon = -1$ if $\ell = k$.  Then, equations \eqref{eqn2} through \eqref{eqn7k} are equivalent to the following homogeneous system of linear equations.

\begin{align*} -\tan\theta (x_3)_\ell +\tan\theta (x_4)_{\ell} - (y_1)_\ell &= 0\\ \cos\theta\sin\theta (x_1)_{\ell} - \cos\theta\sin\theta (x_4)_{\ell} +(\sin^2\theta - \cos^2\theta)(y_1)_{\ell} + \cos\theta \sin\theta (y_3)_{\ell} &= 0\\ \tan\theta (x_2)_\ell -(y_2)_\ell &= 0 \\ \sqrt{3}(x_2)_\ell + \epsilon (x_3)_\ell &= 0\\ \sin^2 \theta (x_1)_\ell  + 2\sqrt{3}(\cos\theta - 1)(x_2)_\ell  + \cos^2\theta (x_4)_\ell &=0\\ 2\sin\theta\cos\theta (y_1)_\ell -2\sqrt{3}\sin\theta (y_2)_\ell + \cos^2\theta (y_3)_\ell &= 0
\end{align*}

Then one can easily compute that all solutions are given as real multiples of \begin{equation}\label{last}\begin{bmatrix} (x_1)_\ell \\ (x_2)_\ell \\ (x_3)_\ell \\ (x_4)_\ell \\ (y_1)_\ell \\ (y_2)_\ell \\ (y_3)_\ell \end{bmatrix} = \begin{bmatrix} -3\cos\theta((2+\epsilon)\cos^2 \theta - 4\cos\theta + 2)\\ -\sqrt{3}\cos\theta \\ 3\epsilon\cos\theta \\ -3 (\cos\theta - 1)( (2+\epsilon) \cos^2\theta  +(\epsilon - 2)\cos\theta - 2)\\ -3\tan\theta((2+\epsilon)\cos^3\theta -4\cos^2\theta + 2)\\ -\sqrt{3} \sin\theta \\ 6\tan^2\theta ((2+\epsilon)\cos^3\theta - 4\cos^2\theta +1)  \end{bmatrix}.\tag{$\ast_3$}\end{equation}

We now note that equation \eqref{eqn1} is equivalent to $y_1(x_1 - x_4) = \tan\theta |x_2|^2$.  In particular, equation \eqref{eqn1} implies that $y_1(x_1  -x_4) > 0 $.  Thus, if we can show that for $\theta \in (0,\pi/6)$, \ref{last} implies $y_1(x_1 - x_4) < 0 $, we will have reached our final contradiction, showing $N_9$ is positively curved when $\theta \in (0,\pi/6)$.

\begin{proposition}

For $\theta \in (0,\pi/6)$, $y_1(x_1-x_4) <  0$.

\end{proposition}

\begin{proof}

We first note that a simple calculation shows $$(x_1)_\ell - (x_4)_\ell = 6 - (6 + 3 \epsilon)\cos\theta.$$

We first prove $y_1(x_1-x_4) < 0$ when $\ell = j$, that is, $\epsilon = 1$.  In this case, $(x_1 - x_4)_{j} = 6- 9\cos\theta$ and this is negative so long as $\cos\theta > \frac{2}{3}$.  Of course, since $\cos(\pi/6) = \frac{\sqrt{3}}{{2}} > \frac{2}{3}$, $(x_1 - x_4)_j < 0$ on $(0,\pi/6)$.

Further, $(y_1)_j = -3\tan\theta (3\cos^3\theta -4\cos^2\theta + 2)$.  The polynomial $3x^3 - 4x^2 + 2$ is clearly positive on the interval $\left(\frac{\sqrt{3}}{{2}}, 1\right)$, so $(y_1)_j < 0$.

It follows that $y_1(x_1 - x_4) = (y_1)_j (x_1 - x_4)_j j^2 = - (y_1)_j(x_1 - x_4)_j < 0$.

\

Finally, we prove $y_1(x_1-x_4) < 0$ when $\ell = k$, that is, $\epsilon = -1$.  Then it is easy to see that $(y_1)_k$ is positive since the polynomial $x^3 - 4x^2 + 1$ is negative on the interval $\left(\frac{\sqrt{3}}{{2}}, 1\right)$.  Further, if $\epsilon = -1$, then $(x_1)_k - (x_4)_k = 6 - 3\cos\theta  > 0$.

Thus, $y_1(x_1 - x_4) = (y_1)_k (x_1 - x_4)_k k^2 = - (y_1)_k (x_1 - x_4)_k < 0$, as claimed.

\end{proof}

\end{document}